\documentclass[12pt]{amsart}
\usepackage{graphicx}
\usepackage{color}
\usepackage{latexsym}
\usepackage{amssymb}                
\usepackage{amsmath}
\usepackage{amsthm}
\usepackage{amscd}
\usepackage{enumerate}
\usepackage{verbatim}
\usepackage{xy}
\xyoption{all}
\theoremstyle{plain}
\newtheorem{thm}{Theorem}
\newtheorem{lem}{Lemma}
\newtheorem{definition}{Definition}

\newtheorem{prop}{Proposition}

\theoremstyle{remark}
\newtheorem{rem}{Remark}

\numberwithin{equation}{section}
\def\ens{\ensuremath}                                     
\newcommand\mtb[1]{\ens{\mathbb{#1}}}                     

\newcommand{\N}{\mtb{N}}	   \newcommand{\Q}{\mtb{Q}}   	\newcommand{\R}{\mtb{R}}	
	   \newcommand{\Z}{\mtb{Z}}	    
\newcommand{\C}{\mtb{C}}    
         \newcommand{\M}{\mtb{M}}
\newcommand\mtc[1]{\ens{\mathcal{#1}}}           
\newcommand{\MM}{\mtc{P}}                	
 
\newcommand{\alp}{\ens{\alpha}}  \newcommand{\bet}{\ens{\beta}}
  \newcommand{\eps}{\ens{\varepsilon}}
\newcommand{\gam}{\ens{\gamma}}	\newcommand{\lam}{\ens{\lambda}}
\newcommand\bld[1]{\ens{\boldsymbol{#1}}}

\newcommand\uv[2]{\scalebox{#1}{#2}} 
   
\newcommand{\dist}{{\bf dist}}                       

\newcommand\hl[1]{{\center \color{red} \ens{\bs\bs\bs} \\}} 
\newcommand\hsp[1]{\mbox{}\hspace{#1mm}} 
\newcommand\vsp[1]{\par \vspace{#1mm}} 
\def\bs{{\bigstar}}                     
\newcommand{\supp}{\uv{.9}{\ens{\mathbf{supp}}}}  
\def\bul{\ens{\bullet}}                 
\newcommand\mbc[1]{}                                                    

\renewcommand{\v}{\vsp}

\newcommand{\h}{\hsp}
\newcommand{\Dst}{\displaystyle}
\newcommand{\Tst}{\textstyle}
\newcommand{\Sst}{\scriptstyle}
\newcommand{\SSst}{\scriptscriptstyle}
\newcommand{\wm}{\mbox{\ens{\SSst \rm  wm}}}
\newcommand{\nwm}{\mbox{\ens{\SSst \rm  nwm}}}
\newcommand{\po}{\mbox{\ens{\SSst\psi=0}}}
\newcommand{\pgo}{\mbox{\ens{\SSst\psi>0}}}
\newcommand{\minus}{\!\setminus\!}
\newcommand{\now}{}
\renewcommand{\ni}{\noindent}
\def\A{\ens{\mtc A}} 
\def\B{\ens{\mtc B}} 
\title[A criterion for weak mixing of induced IETs {\h5 \now}]
{A condition for weak mixing \\ of induced IETs {\h5 \now}}
\author[M. Boshernitzan {}]{Michael Boshernitzan {}}

\address{Department of Mathematics, Rice University, Houston, TX~77005, USA}

\email{michael@rice.edu}
\begin{document}
\maketitle
\vsp5
\begin{abstract}
Let  $f\colon X\to X$, $X=[0,1)$,  be an ergodic IET (interval exchange transformation)
relative to the Lebesgue measure on $X$.  Denote by $f_t\colon X_t\to X_t$ the IET 
obtained  by inducing  $f$  to the subinterval $X=[0,t)$, $0<t<1$. We show that 
\[
X_{\wm}=\{0<t<1\mid  f_{t}  \text{ is weakly mixing}\}
\]
is a residual subset of  $X$  of full Lebesgue measure.
The result is proved by establishing a generic Diophantine sufficient 
condition on $t$  for $f_{t}$ to be weakly mixing.
\end{abstract}
\section{IETs: minimality, ergodicity and mixing}
Denote by $\lam$ the Lebesgue measure on the real line $\R$. Denote by $\Z$, 
$\N=\{k\in\Z\mid k\geq1\}$  the sets of integers and of natural numbers, respectively.
We write $\sharp(S)$  for the cardinality of a set $S$.

An IET (interval exchange transformation) is a pair  $(X,f)$ where $X=[0,b)\subset\R$ 
is a bounded interval, and $f$ is a right continuous bijection $f\colon X\to X$ of it 
with a finite set  $D$ of discontinuities and such that  $f'(x)=1$ for all\,  
$x\in X\minus D$. (The last conditions means that $f$  is a local translation 
at every its continuity point).
We often refer to the map $f$  itself as an IET.

Let $\sharp(D)=r-1$, $r\geq2$. Then  without loss of generality                       \mbc{eq:disc\\1.1}
\begin{equation}\label{eq:disc}
    D=\{d_k\}_{k=1}^{r-1}; \quad d_0=0<d_1<\ldots<d_{r-1}<b=d_r.
\end{equation}
(The conventions  $d_0=0$, $d_r=b$ are used for convenience. 
Note that $0,b\notin D$).

An IET $(X,f)$ (or $f$) with $\sharp(D)=r-1$ is also called (more specifically) 
an $r$-IET referring to the fact that $f$  exchanges the $r$~intervals  
$X_{k}=[d_{k-1},d_{k})\subset X$,  $1\leq k\leq r$,  
according to some permutation  $\rho\in S_{r}$.

An $r$-IET  is completely determined by this permutation $\rho\in S_{r}$ and 
the lengths $\lambda_{k}=d_{k}-d_{k-1}>0$, $1\leq k\leq r$, of exchanged 
subintervals~$X_{k}$. Thus $r$-IETs  can be identified with
pairs $({\vec\lam},\rho)$  where $\vec\lam\in (\R^{+})^r$ and  
$\rho\in S_{r}$:   $(X,f)=({\vec\lam},\rho)$.

A permutation $\rho\in S_{r}$ is called {\em irreducible}\/ if
$\rho(\{1,2,\ldots,k\})\neq(\{1,2,\ldots,k\})$,  for all $k<r$.

An IET $(X,f)$  is called minimal if all $f$-orbits are dense in $X$. Note that 
the irreducibility of  $\rho$  is a necessary condition for $f$  to be minimal 
because otherwise  $X$  splits into two  \mbox{$f$-invariant} subintervals.

Keane \cite{Kea_IET} proved that if $r\geq2$ and if $\rho\in S_{r}$ is irreducible 
then the IET $({\vec\lam},\rho)$  is minimal provided~that the lengths  $\lam_{k}$ 
of exchanged intervals are linearly independent over the rationals 
(a generic assumption on~$\vec\lam$).

Masur \cite{Mas} and Veech  \cite{Ve_Gauss} independently proved the following 
result  (conjectured by Keane in \cite{Kea_IET}): 
If  $r\geq2$ and $\rho\in S_{r}$  is irreducible, then for Lebesgue almost all 
$\vec \lam\in(\R^{+})^r$ the IET $({\vec\lam},\rho)$ is uniquely ergodic (all its orbits 
are uniformly distributed). Alternative approaches to Keane's conjecture were later given 
by Rees~\cite{Rees}, Kerchoff \cite{Ker} and Boshernitzan \cite{Bo_con1}.
(Boshernitzan exhibited some Diophantine conditions, including a generic one, 
Property P, for unique ergodicity of IETs, see \cite{Bo_con1}, \cite{Bo_rank2} and 
\cite{Ve_Bos}, an improvement by Veech. The idea to use a suitable generic 
Diophantine condition to establish a metric result is also central in the present paper).

Avila and Forni \cite{AvFo} proved that Lebesgue almost all IETs are weakly mixing 
(assuming that $\rho$  is irreducible and not a rotation). Partial results in this 
directions were obtained earlier by Katok and Stepin \cite{KaSt} who established weak 
mixing of generic $3$-IETs and Veech \cite{Ve_metric} who proved generic weak mixing 
for a large class of permutations (Veech permutations).  
It was later shown by Boshernitzan and Nogueira \cite{BoNo} that if an IET 
$f=({\vec\lam},\rho)$  satisfies property P (a generic condition used in \cite{Bo_con1}),
or even a weaker condition \cite{Ve_Bos},  and if~$\rho$~is a Veech permutation, then  
$f$  is weakly mixing.

Note that Katok proved that IETs are never (strongly) mixing \cite{Ka_nomix}. 
On the other hand, Chaika \cite{chaika_mix1} constructed a $4$-IET which is 
topologically mixing. Boshernitzan and Chaika \cite{BoCh} showed  that $3$-IETs are 
never topologically mixing.
\v3
\section{The results.}
In what follows let $(X,f)$  be a fixed aperiodic $r$-IET, $X=[0,b)$, $r\geq2$.
(Aperidicity of $f$ means absence of $f$-periodic points).

Denote by $X_{t}$, $0<t<b$, the subinterval $[0,t)\subset X$ and by $f_{t}$ the IET 
obtained by inducing $f$  to $X_{t}$. It is well known that each $(X_{t},f_{t})$ 
is an $s$-IET  with  $s=s(t)\leq n+1$ (see \cite{CFS} or \cite{Kea_IET}). 
(In fact, $s(t)\geq2$  due to the aperiodicity assumption).

The central result of the paper is the following.                                     
Recall that $\lam$ stands for the Lebesgue measure on  $X$.                              \mbc{thm:gwm\\1}
\begin{subequations}
\begin{thm}\label{thm:gwm}
Let  $(X,f)$  be a $\lam$-ergodic IET, $X=[0,b)$.  Then the set 
\begin{equation}\label{eq:xwm}
X_{\wm}=X_{\wm}(\lam)=\{0<t<b\mid f_{t} \text{ is weakly mixing (relative $\lam$)}\,\}
\end{equation}
is a residual set of full measure: $\lam(X_{\wm})=1$.
\end{thm}
Because of the above result, we refer to the complement set  
\begin{equation}\label{eq:xnwm}
X_{\nwm}=(0,b)\minus X_{\wm}
\end{equation}
as ``the exceptional set for (weakly mixing induction of) (X,f)''.\v2
\end{subequations}
\begin{rem}
Two additional versions of Theorem \ref{thm:gwm} (for minimal but not uniquely ergodic IETs)
are given in Section~\ref{sec:2vers}.
\end{rem}


We need some notation. Recall that  $D=\{d_k\}_{k=1}^{r-1}$ stands 
for the set of discontinuities of~$f$. Denote by
\[
D_{0}=D\cup\{0,b\}=\{d_k\}_{k=0}^r
\]
the set of $(r+1)$ points in \eqref{eq:disc}, and denote by
\begin{equation}\label{eq:dp}
D'\!=\!\!\Tst\bigcup_{\SSst k=-\infty}^{\infty}\limits\!f^k(D)
\end{equation}
the set of points whose orbits hit $D$. 

The aperiodicity of $f$ implies that  $D'$  is a dense countable subset 
of $X=[0,b)$ contai\-ning~$0$  (see e.g. \cite[Section 2]{Bo_rank2}).

For $x\in X$ and $n\geq1$,   we set  
\begin{subequations}\label{eqs:r}
\begin{align}
\rho(x)&=\dist(x,D_0)=\min_{0\leq k\leq r}\limits |x-d_{k}|\label{eq:ra}\\
\rho_{n}(x)&=\min_{-n\leq k\leq n-1}\rho(f^k(x))\label{eq:rb}\\
\Delta_{n}(x)&=\min_{\substack{|p|,|q|\leq n\\ p\neq q}}
    |f^{p}(x)-f^{q}(x)|\label{eq:rc}\\[-5mm]
\intertext{and}
\rho'_{n}(x)&=\min(\rho_{n}(x), \tfrac12\Delta_{n}(x)).\label{eq:rd}
\end{align}
\end{subequations}
Note that for all\, $x\!\in\!D'$ both $\rho_{n}(x)$ and $\rho'_{n}(x)$  vanish for 
large $n$, while  for\, $x\!\in\! X\!\setminus\! D'$ we have\,  
\mbox{$\rho_{n}(x)\!\geq\!\rho'_{n}(x)\!>\!0$}  due to the aperiodicity of $f$.

Important interpretations of the values $\rho_{n}(x)$ and $\rho'_{n}(x)$ are given  
by Propositions \ref{prop:rho1} and \ref{prop:rho2} in the next section.
The following two functions                                                        
\begin{subequations}
\begin{align}
&\phi\colon X\to\R; && \phi(x)=\phi_{f}(x)=\limsup_{n\to\infty} \ n\rho_{n}(x);
\label{eq:fi} \\[-2mm]
\intertext{and}
&\psi\colon X\to\R; && \psi(x)=\psi_{f}(x)=
\limsup_{n\to\infty} \ n\rho'_{n}(x);\label{eq:psi}
\end{align}
will play central role in the paper. (We usually suppress the dependence                  \mbc{thm:cwm\\2}
on $f$ by writing $\phi,\psi$ rather than $\phi_{f},\psi_{f}$ due to the 
standing assumption that $(X,f)$ is a fixed aperiodic IET).      
\end{subequations}   
\begin{thm}[\bf A sufficient condition for weak mixing of $f_{t}$] \label{thm:cwm}
Let  $f\colon\! X\to\! X$  be a \mbox{$\lam$-ergodic} IET, $X\!=\![0,b)$.  If \  $\psi(t)>0$ 
for some  $t\in(0,b)$,  then the induced IET\, \mbox{$f_t\colon X_{t}\!\to\! X_{t}$}, 
$X_{t}=[0,t)$, is weakly mixing.
\end{thm}
\begin{subequations}
For an aperiodic IET\,   $f\colon X\to X$,  $X=[0,b)$, the set
\begin{equation}\label{eq:crit}
X_{\po}=\big\{t\in (0,b)\mid \psi(t)=0\big\}
\end{equation}
will be referred as ``the critical set for $(X,f)$''. We also adopt similar notation
\begin{equation}\label{eq:crit2}
X_{\pgo}=(0,b)\!\setminus\!X_{\po}=\big\{t\in (0,b)\mid \psi(t)>0\big\}
\end{equation} 
\end{subequations}
for the complement of this set.

Theorem \ref{thm:cwm} claims that for an \lam-ergodic  IET\,  $(X,f)$\,  the inclusion
$X_{\nwm}\!\subset\!X_{\po}$  takes place. (In other words, every exceptional point 
must be critical, see \eqref{eq:xnwm} and   \eqref{eq:crit2}). 

Equivalently, $X_{\pgo}\!\subset\!X_{\wm}$ (see \eqref{eq:xwm} and \eqref{eq:crit}).

The importance of Theorem \ref{thm:cwm} is twofold. 
First, it provides a {\em generic} condition   (Theorem~\ref{thm:mgg})  
sufficient to establish Theorem \ref{thm:gwm} claiming that the exceptional set 
$X_{\nwm}$ is small in both measure and topology categories. 

Secondly, it allows (under certain Diophantine conditions on\, $t$\, and\, $f$) to 
establish more delicate infor\-ma\-tion on the ``smallness'' of the exceptional set  
$X_{\nwm}$. In particular, there are examples of ergodic IETs $(X,f)$ for which one 
can show that the critical set  $X_{\po}$  coincide with $D'$,  and hence  
$X_{\nwm}\subset D'$  is at most countable (see Section \ref{sec:last}).

An IET $(X,f)$ is called {\em persistently weakly mixing}\/ if  $X_{\nwm}=\emptyset$.              \mbc{25}
It is easy to see that there are no persistently weakly mixing $r$-IETs
with $r<4$  (because then  $f_{t}$  become rotations for a countable set
of $t$).

We believe that there exist persistently weakly mixing  IETs.

{\bf Question}. What is the ``size''  (in the sense of measure, category and cardinality)
of $X_{\nwm}$  for ``most''  $4$-IETs  with permutation $\rho=(4321)$?

The answers for the same questions for $r$-IETs  with $r=2$ or $3$ are known
(see Section \ref{sec:last} for the answers without proofs).

\begin{thm}\label{thm:mgg}
Let  $f\colon X\to X$  be an aperiodic IET, $X=[0,b)$.                                 \mbc{thm:mgg\\3}
Then the critical set $X_{\po}$  is meager and has Lebesgue measure $0$.     
\end{thm}
Theorem \ref{thm:gwm} follows immediately from Theorems \ref{thm:cwm}                  
and \ref{thm:mgg}, the proofs of which are presented in 
Sections~\ref{pr:thm:cwm} and \ref{sec:pr:thm:mgg}, respectively.

\section{Some notation, terminology and lemmas.}\label{sec:def}                       \mbc{sec:def\\3}
The discussion in this section continues under the assumption that $(X,f)$  
is a fixed aperiodic $r$-IET, $X=[0,b)$, $r\geq2$. Note that  $f^{-1}$ (the 
compositional inverse of $f$)  is also an aperiodic $r$-IET on $X$.     

\begin{definition}\label{def:fbasic}
An open subinterval $Y\subset X$ is called $f$-basic if the following                \mbc{def:fbasic\\1}
equivalent conditions are met: 

{\em(b1)} $f|_Y$ is a translation;

{\em(b2)}  $f|_Y$ is continuous;

{\em(b3)} $Y\cap D=\emptyset$.\\
Given an $f$-basic interval  $Y$, we write $\overset{\SSst+}Y(f)$
for the translation constant \ $f|_{_Y}(y)-y$.
\end{definition}
Observe that if an interval $Y\subset X$ is  $f$-basic then  $f(Y)$ is 
an $f^{-1}$-basic interval.
\begin{definition}\label{def:fst}
A sequence $\vec Y=(Y_{k})_{k=1}^{n}$ of subsets of $X$  is called an               \mbc{def:fst\\2}
$f$-stack if the following conditions are met:

{\em(s1)} Each of the sets $Y_{k}$, $1\leq k\leq n-1$, is an
    $f$-basic interval;
    
{\em(s2)} $f(Y_{k})=Y_{k+1}$, for $1\leq k\leq n-1$.\\[2mm]
An $f$-stack $\vec Y=(Y_{k})_{k=1}^{n}$ is called\/ {\em distinct} if
the sets  $Y_k$  are pairwise disjoint.
Given an $f$-stack\, $\vec Y=(Y_{k})_{k=1}^{n}$, we use the following
terminology:\\[1mm]
\h{18}
\begin{tabular}[t]{lll}
\bul\,{\bf The  width \hfill of \ $\vec Y$:} &&   $\omega(\vec Y)=\lam(Y_1)$ 
       (in fact, all\, $\lam(Y_k)$ are equal);\\
\bul\,{\bf The support \hfill of \  $\vec Y$:} &&
       $\supp(\vec Y)={\Tst\bigcup_{k=1}^{n}} Y_{k}\subset X$;\\
\bul\,{\bf The length \hfill of \ $\vec Y$:} && $h(\vec Y)=n$;\\
\bul\,{\bf The measure \hfill of \ $\vec Y$:} && 
    $\lam(\vec Y)=\lam(\supp(\vec Y))$.\\
\end{tabular}\\[1mm]
Note that if \ $\vec Y=(Y_{k})_{k=1}^{n}$ is a distinct $f$-stack,
then  $\lam(\vec Y)=\omega(\vec Y)h(\vec Y)$.
\end{definition}

  Observe that if $(Y_{k})_{k=1}^{n}$ is an $f$-stack then 
the inverted sequence $(Y_{n+1-k})_{k=1}^{n}$ forms an $f^{-1}$-stack;
in particular, the last set $Y_n$ must also be an open subinterval of $X$
(but not necessarily an $f$-basic one).


For $x\in\R$ and $\eps>0$, denote by $B_{\eps}(x)=(x-\eps,x+\eps)$ the 
$\eps$-neighborhood of $x\in \R$. 

\begin{prop}\label{prop:rho1}
Let  $(X,T)$ be an aperiodic IET. For $\eps>0$, $x\in X\!\setminus\! D'$              \mbc{prop:rho1 1}
and $n\geq1$, the following three conditions are equivalent:

{\em(\ref{prop:rho1}a)} The sequence of intervals 
$\Big(B_{\eps}\big(T^{k}(x)\big)\Big)_{k=-n}^{n}$ forms an $f$-stack 
(of length $(2n+1)$).

{\em(\ref{prop:rho1}b)}\, $\eps\leq \rho_{n}(x)$.

{\em(\ref{prop:rho1}c)}\, There exists an $f$-stack  
$(Z_{k})_{k=-n}^{n}$ with $Z_0=B_{\eps}(x)$.
\end{prop}
\begin{proof}
Follows from the definition of $\rho_{n}(x)$  
(see \eqref{eq:ra}).

\end{proof}
\begin{prop}\label{prop:rho2}
Under the assumptions and notations as in Proposition\/ {\em\ref{prop:rho1}},            \mbc{prop:rho2 2}
assume that the equivalent conditions {\em(\ref{prop:rho1}a)}, 
{\em(\ref{prop:rho1}b)} and  {\em(\ref{prop:rho1}c)}  hold. 
Then the following three conditions are equivalent:

{\em(\ref{prop:rho2}a)} The $f$-stack  
    $\Big(B_{\eps}\big(T^{k}(x)\big)\Big)_{k=-n}^{n}$  is distinct;

{\em(\ref{prop:rho2}b)}\,  $\eps\leq \rho'_{n}(x)$.

{\em(\ref{prop:rho2}c)}\, There exists a distinct $f$-stack  
    $(Z_{k})_{k=-n}^{n}$ with $Z_0=B_{\eps}(x)$.
\end{prop}
\begin{proof}
Follows from the definition of $\rho'_{n}(x)$  (see \eqref{eq:rd}).

\end{proof}

The following lemma  will be used in the proof of Theorem \ref{thm:mgg}. 
(Various versions of it are well known).
\begin{lem}\label{lem:1} 
Let  $(X,f)$  be a minimal\, $r$-IET, $X=[0,b)$. Then for every $N$, there exists            \mbc{lem:1 1}
a distinct $f$-stack  $\vec Y$  of length at least  $N$ and of measure at least  
$\frac br$.
\end{lem}
\noindent {\bf Remark}. The above lemma holds under the weaker assumption that $(X,f)$
is aperiodic (rather than minimal). We do not use the stronger version, and its proof 
is not included.
\begin{proof}[{\bf Proof of Lemma \ref{lem:1}}] Pick a small subinterval $Y\subset X$, 
$0<\lam(Y)<\frac b{rN}$, so that the induced map\,  $g$\,  on\,~$Y$ is an  $s$-IET, 
with some  $2\leq s\leq r$.  
Let  $Y_k\subset Y$,  $1\leq k\leq s$,  be the subintervals of  $Y$  exchanged 
by  $g$:\,     $g(Y_k)=f^{n_{k}}(Y_k).$
By minimality, the images $f^n(Y_k)$ cover $X=[0,b)$ before returning to $Y$.

More precisely,  the family of subintervals
$\big\{f^n(Y_k)\,\big|\,1\leq k\leq s,\ 0\leq n\leq n_k \big\}$
partitions the interval\,  \mbox{$X=[0,b)$}. 
Thus\, $\lam\Big(\bigcup_{n=0}^{n_k} f^n(Y_k)\Big)\geq\tfrac bs\geq\tfrac br$,\,
for some  $k\in[1,s]$.
To satisfy the conditions of the Lemma \ref{lem:1}, one takes  
$\vec Y=\big(f^{n}(Y_{k})\big)_{n=0}^{n_{k}}$.\v{-3}

\end{proof}       
\v{-1}
We would need the following notation. For an open finite interval  $Y$, denote by
$\Theta(Y)$  the middle third subinterval of $Y$ defined as the interval with the same
center but\, $3$\, times shorter: 
\begin{equation}\label{eq:Theta}
\Theta\big(B_{\eps}(x)\big)=B_{\eps/3}(x); 
\h5 \Theta\big((a,b)\big)=\big(\tfrac{2a+b}3,\tfrac{a+2b}3\big); \h5
\lam(\Theta(Y))=\tfrac13\lam(Y).
\end{equation}¥

\section{Proof of Theorem \ref{thm:mgg}}\label{sec:pr:thm:mgg}
One easily validates $f$-invariance of $\psi$:  $\psi(x)=\psi(f(x))$,             \mbc{sec:pr:thm:mgg\\3}
So the Borel $f$-invariant set  $X_{\pgo}$  must have Lebesgue measure 
either $0$ or $1$,  in view of the ergodicity of $f$.

Since $(X,f)$ is an ergodic IETs, it is minimal, so Lemma \ref{lem:1} applies. It 
follows that there exists a sequence $(\vec Y_{n})_{n\geq1}$  of distinct $f$-stacks 
with lengths $h(n)\colon\!\!=h(\vec Y_{n})$ approaching infinity and measures  
$\lam(\vec Y_{n})\geq \frac br$ (see Definition \ref{def:fst}).   Let
\[
\vec Y_n=\big(Y_{n,k}\big)_{k=1}^{h(n)}=(Y_{n,1}, Y_{n,2}, \ldots, Y_{n,h(n)}), 
     \qquad n\in \N.
\]
We may assume that all\, $h(n)\geq6$. Let $p(n)=\Big[\frac{h(n)}3\Big]$ and  
$q(n)=h(n)-p(n)+1$.

Consider the following sequence of distinct $f$-stacks $(\vec Z_{n})_{n\geq1}$:
\[
\vec Z_n=\Big(\Theta\big(Y_{n,k}\big)\Big)_{k=p(n)+1}^{q(n)-1}=
\Big(\Theta\big(Y_{n,p(n)+1}\big), \Theta\big(Y_{n,p(n)+2}\big), \ldots, 
\Theta\big(Y_{n,q(n)-1}\big)\Big), \qquad n\in \N,
\]
and set  
\[
Z_n=\supp(\vec Z_n)\subset X,  \qquad n\in \N,
\]
(see Definition \ref{def:fst} for notation) and
\[
Z = \bigcap_{n=1}^{\infty}\bigg(\bigcup_{k=n}^{\infty} Z_k\bigg)=
   \{z\in X\mid z\in Z_n, 
   \text{ for infinitely many }n\geq1\}.
\]

The set $Z$  is clearly a residual subset of $X$. (In fact, it is a dense
$G_{\delta}$ subset of $X$).\v2

Observe the following inequalities for the length and the width of $\vec Z_n$:
\[
h(\vec Z_n)=q(n)-p(n)-1\geq \tfrac{h(n)}3,  \quad 
\omega(\vec Z_n)=\tfrac13\omega(\vec Y_n)  \qquad (n\in \N).
\]
It follows that
\[
  \lam(Z_n)=\lam(\vec Z_n)\geq \tfrac19 \lam(\vec Y_n)\geq\tfrac b{9r}
   \qquad (n\in \N),
\]
and therefore
\[
   \lam(Z)\geq \limsup_{n\to\infty} \lam(\vec Z_n)\geq\tfrac b{9r}.
\]

\begin{lem}\label{lem:psi}
   $\psi(z)>0$  for all  $z\in Z$.
\end{lem}

Since the function $\psi$ is easily seen to be Borel measurable and $f$-invariant  
($\psi(x)=\psi(f(x))$, for all $x\in X$), the ergodicity of $f$ implies that 
$\lam(X_{\pgo})\in\{0,1\}$. The proof of Lemma \ref{lem:psi} would imply that  
$\lam(X_{\pgo})=1$, completing the proof of Theorem  \ref{thm:mgg}.

\begin{proof}[\bf Proof of Lemma \ref{lem:psi}] Let $z\in Z$. Then there exists an 
increasing sequence  of positive integers  $(n_i)_{i=1}^{\infty}$  such that  
$z\in Z_{n_i}$ where
\v{-5}
\[
 Z_{n_i}=\ \supp(\vec Z_{n_i})=\h{-3}\bigcup_{k=p(n_i)+1}^{q(n_i)-1} 
   \Theta(Y_{n_i,k})\ \subset\ \supp(\vec Y_{n_i})=\bigcup_{k=1}^{h(n_{i})} \ 
   Y_{n_i,k}, \h6 \text{for all }\, i\geq1.
\]
Set\, $\eps_{i}\colon\!\!=\tfrac12\omega(\vec Z_{n_i})=\tfrac16\omega(\vec Y_{n_i})$.
Note that for every $i\geq1$  the sequence\,
\[
\vec B_i\colon\!\!=\Big(f^k\big(B_{\eps_{i}}(z)\big)\Big)_{k=-p(n_i)}^{p(n_i)}
\]
forms a distinct $f$-stack due to the fact that  $\vec Y_{n_i}$  does.
By Proposition \ref{prop:rho2}, \ $\eps_{i}\leq\rho'_{p(n_i)}(z)$.

Since $\Dst\lim_{i\to\infty} n_i=\infty$,  we get both 
$\Dst\lim_{i\to\infty} h(n_i)=\infty$ and $\Dst\lim_{i\to\infty} p(n_i)=\infty$. 
One concludes that for all large $i$\,:
\begin{align*}
\rho'_{p(n_{i})}(z)\cdot p(n_i)\geq
   \eps_{i}\cdot h(n_i)\cdot\tfrac{p(n_i)}{h(n_i)}=
   \tfrac16\cdot \lam(\vec Y_{n_i})\cdot\tfrac{p(n_i)}{h(n_i)}> 
   \tfrac16 \cdot\tfrac br\cdot \tfrac14=\tfrac b{24r}.
\end{align*}
The proof of Lemma \ref{lem:psi} (and hence of Theorem \ref{thm:mgg}) is 
completed by  direct estimation (see~\eqref{eq:psi}): 
\[
  \psi(z)=\limsup_{n\to\infty}\limits\,n\rho'(n)\geq \tfrac b{24r}>0.
\]
\end{proof}

\section{Proof of Theorem \ref{thm:cwm}} \label{pr:thm:cwm}
Denote by $S^{1}=\{z\in\C\mid |z|=1\}=\{e^{it}\mid t\in[0,2\pi)\}$\,                 \mbc{pr:thm:cwm\\2}
the unit circle in the complex plane. 

Theorem \ref{thm:cwm} is derived from the following proposition. 
\begin{prop}\label{prop:F}
Let $(X,f)$ be an aperiodic IET, $X=[0,b)$. Let $1\neq\theta\in S^{1}$.              \mbc{prop:F 3}
Assume that\, $\psi(t)>0$, for some $t\in(0,b)$.
Then the equation
\begin{equation}\label{eq:F}
F(f(x))=
\begin{cases}
\theta\cdot F(x) & \text{if }\, x<t\\
F(x)  & \text{if }\, x\geq t
\end{cases}
\end{equation}¥
has no (Lebesgue) measurable solutions $F\colon X\to S^{1}$.
\end{prop}
\begin{proof}[\bf Proof of Theorem \ref{thm:cwm}]
Since $(X,f)$ is ergodic, all induced maps $(X_{t},f_{t})$ also are. If some IET 
$(X_{t},f_{t})$  fails to be weakly mixing, it has a nontrivial eigenvalue  
$\theta\in S^{1}$, $\theta\neq 1$. Select an eigenfunction $G\colon X_{t}\to\C$ 
corresponding to $\theta$ so that $G(f(x))=\theta\cdot G(x)$, for $x\in X_{t}$.
The ergodicity of $f_{t}$ implies that  $|G|$  must be a constant which 
(without loss of generality) is assumed to be $1$. Thus $G(X_{t})\subset S^{1}$.

Let  $g=f^{-1}$  be the compositional inverse of $f$. Define  $F(x)=G(g^{k(x)}(x))$  
with  $k(x)=\min \big(N_{t}(x)\big)$ where the set\,  
$\N_{t}(x)\colon=\{k\geq 0\mid g^{k}(x)\in X_{t}\}$  is not empty because both 
$f$, $f^{-1}$  are ergodic and hence minimal.
The constructed function  $F\colon X\to S^{1}$  is measurable (because $G$ is) and 
is easily seen to satisfy \eqref{eq:F}. This contradicts the conclusion of 
Proposition \ref{prop:F}, completing the proof of Theorem \ref{thm:cwm}.

\end{proof}
\section{Proof of Proposition \ref{prop:F}}
The proof goes by contradiction. Assume to the contrary that there are $\theta$ and 
$F$ satisfying the conditions of Proposition \ref{prop:F} and that, in particular, 
\eqref{eq:F} holds for some  $t\in(0,b)$ such that  
\[
\psi(t)=\limsup_{n\to\infty}\limits\, n\rho'_{n}(t)>0.
\]
Set $\eps_k=\frac{\psi(t)}{2k}$, for $k\geq1$.
Then there exists an infinite subset\,  $\M\subset \N$\, 
of natural numbers such that\,  $\rho'_n(t)>\eps_n>0$\,  for all\, $n\in \M$.

The statement of the following lemma follows from Proposition \ref{prop:rho2}.
\begin{lem}\label{lem:ds}
The sequence $\big(B_{\eps_n}(f^k(t))\big)_{k=-n}^n$ forms a                            \mbc{lem:ds\\3}          
distinct $f$-stack for every integer $n\in\M$.
\end{lem}

\subsection*{Some notation} 
For $n\in \M$ and $|k|\leq n$, set the following open subintervals of $X$:
\begin{subequations}
\begin{align*}
A_n^k&=(f^k(t)-\eps_n,f^k(t)+\eps_n)=B_{\eps_n}(f^k(t));\\
B_n^k&=(f^k(t),f^k(t)+\eps_n);\\
C_n^k&=(f^k(t)-\eps_n, f^k(t)).
\end{align*}
\end{subequations}

Set the constants (all lying in $D_1=\big\{z\in \C\ \big|\ |z|\leq1\big\}$):
\begin{equation}\label{eq:alpkn}
\alp^k_n=\A(A^k_n), \h9  \bet^k_n=\A(B^k_n), \h9 \gam^k_n=\A(C^k_n),
\end{equation}
where
\begin{equation}\label{eq:aver}
\A(Y)=\!\tfrac1{\lam(Y)}\int_{Y} F(x)\,dx
\end{equation}
stands for the average of the function $F$ over a subinterval $Y\subset X$. 

Since  $F$ is $S^1$-valued, the constants in \eqref{eq:alpkn}
lie in the unit disc\, $D_1=\big\{z\in\C\ \big|\ |z|\leq1\big\}$.

By a passing to an infinite subset of $\M$, 
we may assume that the following six limits 
\begin{equation}\label{eq:alpi}
\alp^i=\!\lim_{\substack{n\in \M\\ n\to\infty}} \alp^{i}_{n}, \h5
\bet^i=\!\lim_{\substack{n\in \M\\ n\to\infty}} \bet^{i}_{n}, \h5
\gam^i=\!\lim_{\substack{n\in \M\\ n\to\infty}} \gam^{i}_{n},  \h9
i\in\{0,1\}
\end{equation}
\v{-1.5}
\ni exist and lie in $D_1$. We show that in fact all six constants
$\alp^{1},\alp^{0},\bet^{1},\bet^{0},\gam^{1}$ and $\gam^{0}$
must lie in the unit circle $S^1=\partial D_1$  (see Lemma \ref{lem:sixcon} below).
This will follow from Lemma~\ref{lem:mod1} below.

Recall that, given an $f$-stack $\vec Y$, we write $h(\vec Y)$ and\, $\lam(\vec Y)$
for the length and the measure of $\vec Y$ (see Definition~{\ref{def:fst}}).

\begin{lem}\label{lem:mod1}
Let $(X,f)$ be an aperiodic IET, $X=[0,b)$. Let\, $\theta\in S^1$ and\, $t\in(a,b)$.          
Assume that a measurable function $F\colon X\to S^1$ satisfies the equation             \mbc{lem:mod1 4}
\eqref{eq:F}. Let \ $(\vec Y_n)_{n=1}^{\infty}$  be a sequence of distinct 
$f$-stacks
\[
\vec Y_n=\big(Y_{n,k}\big)_{k=1}^{h_{n}}=(Y_{n,1},Y_{n,2},\ldots,Y_{n,h_{n}}),
\]
satisfying the following two conditions:

\h5{\em(\ref{lem:mod1}A)} $\lim_{n\to\infty}\limits h_n=\infty$  (the lengths of 
stacks  $h_n=h(\vec Y_n)$ approach infinity);
\v{-1}

\h5{\em(\ref{lem:mod1}B)} $\liminf_{n\to\infty}\limits\, \lam(\vec Y_n)>0$
          (the measures of stacks\,  
          $\lam(\vec Y_n)=\!\sum_{k=1}^{h(n)}\limits \lam(Y_{n,k})$\, 
          stay away from $0$);
      
\h5{\em(\ref{lem:mod1}C)}\, $t\notin Y_{n,k}$, for all\, $n$\, and\, $k\in[1, h_{n}]$.
\v2
\ni Then \ $\lim_{n\to\infty}\limits|\A(Y_{n,1})|=1$  (for notation see \eqref{eq:aver}).
\end{lem}
\ni{\bf Remark}. In the above lemma, under the condition  (\ref{lem:mod1}C)  alone 
(i.e., without assuming  (\ref{lem:mod1}A) and (\ref{lem:mod1}B))  we obviously have
\[
|\A(Y_{n,1)}|=|\A(Y_{n,2)}|=\ldots=|\A(Y_{n,h_n)}|, \quad \text{ for all }\, n\geq1,
\]
in view of the equation \eqref{eq:F}  the function $F$ satisfies. In particular, 
the conclusion \ $\lim_{n\to\infty}\limits|\A(Y_{n,1})|=1$ in Lemma \ref{lem:mod1} 
is equivalent to the relation \ $\lim_{n\to\infty}\limits|\A(Y_{n,h_n})|=1$.

\begin{proof}[\bf Proof of Lemma \ref{lem:mod1}]
For subintervals $Y\subset X$,  set \ 
$\B(Y)=\int_{Y} |F(x)-\A(Y)|\,dx$ \
where  $\A(Y)$ stands for the average of $F$ over $Y$ (see \eqref{eq:aver}).
Then  \ $\lim_{n\to\infty}\limits \Big(\sum_{k=1}^{h_n}\,\B(Y_{n,k})\Big)=0$ \
and
\[
\lam(Y_{n,1})=\lam(Y_{n,2})=\ldots=
\lam(Y_{n,h_n})\leq\tfrac1{h_n}\to 0 \qquad \text{(as }n\to\infty),
\]
because of the assumption (\ref{lem:mod1}A).
Since $F$  satisfies \eqref{eq:F}, we have
\[
\B(Y_{n,1})=\B(Y_{n,2})=\ldots=\B(Y_{n,h_n}),
\]
and hence \
$\lim_{n\to\infty}\limits\,\B(Y_{n,1})\cdot h_{n}=0$. 
(Here we use the assumption (\ref{lem:mod1}C)).  It follows that  
\[
\lim_{n\to\infty}\bigg(\Big(\tfrac1{\lam(Y_{n,1})}\cdot\B(Y_{n,1})\Big)\cdot
\big(h_n\cdot\lam(Y_{n,1}) \big)\bigg)=
\lim_{n\to\infty}\bigg(\Big(\tfrac1{\lam(Y_{n,1})}\cdot\B(Y_{n,1})\Big)\cdot
\lam(\vec Y_n) \bigg)=0.
\]

\ni In view of the assumption (\ref{lem:mod1}B), we conclude that
\[
\lim_{n\to\infty}\tfrac1{\lam(Y_{n,1})}\cdot\B(Y_{n,1})=
\lim_{n\to\infty} \ \tfrac1{\lam(Y_{n,1})} \int_{Y_{n,1}} |F(x)-\A(Y_{n,1})|\,dx=0.
\]
Let  $\eps>0$  be given. Then for all sufficiently large $n$, one can select  
$x_{n}\in Y_{n,1}$  so that $F(x_n)\in S^1$ \ and \ $|F(x_n)-\A(Y_{n,1})|<\eps$. 
This implies that\,  $\big||\A(Y_{n,1})|-1\big|<\eps$, and, since  $\eps>0$ is arbitrary,  
$\lim_{n\to\infty}\limits |\A(Y_{n,1})|=1$, completing the proof of Lemma \ref{lem:mod1}.

\end{proof}
\v{-5}
\begin{lem}\label{lem:sixcon}
The six constants  $\alp^{1},\alp^{0},\bet^{1},\bet^{0},\gam^{1}$ and $\gam^{0}$
(see  \eqref{eq:alpkn})   lie in $S^{1}$.  \    \mbc{lem:sixcon\\ \ref{lem:sixcon}}                   
\end{lem}
\begin{proof} {\bf Case of $\bld{\alp^{1}}$}.  Recall that 
$\alp^1=\!\lim_{\substack{n\in \M\\ n\to\infty}}\limits \alp^{1}_{n}$  where
$\alp^1_n=\A(A^1_n)$. \v{-2}

We claim that the conditions of Lemma \ref{lem:mod1} are fulfilled
with\,  $\big(\vec Y_{n}\big)_{n\geq1}=\Big(\big(A^k_n\big)_{k=1}^n\Big)_{n\in \M}$.
Indeed, in this case $\vec Y$ is a distinct $f$-stack in view of Lemma \ref{lem:ds}.
We also have 
\[
h(\vec Y_n)=n \qquad \text{\small and} \qquad\lam(\vec Y_n)=2\eps_{n}h(\vec Y_n)=2\eps_{n}n=\psi(t)=2\eps_1
\]
for $n\in \M$. It follows from Lemma \ref{lem:mod1} that  $|\alp^1|=1$.

\ni{\bf Case of $\bld{\alp^0}$}. Similar argument. We set
 $\big(\vec Y_{n}\big)_{n\geq1}=\Big(\big(A^k_n\big)_{k=-n+1}^0\Big)_{n\in \M}$
 and take in account remark following Lemma \ref{lem:mod1} to get   $|\alp^0|=1$.
 
\ni{\bf Case of $\bld{\bet^1}$}. We set  
$\big(\vec Y_{n}\big)_{n\geq1}=\Big(\big(B^k_n\big)_{k=1}^n\Big)_{n\in \M}$ and 
in the same way apply Lemma \ref{lem:ds} to get $|\bet^1|=1$.

\ni{\bf Case of $\bld{\bet^0}$}. We set  
$\big(\vec Y_{n}\big)_{n\geq1}=\Big(\big(B^k_n\big)_{k=-n+1}^0\Big)_{n\in \M}$ and 
in the same way apply Lemma \ref{lem:ds} to get $|\bet^0|=1$.

\ni{\bf Cases of $\bld{\gam^1}$ and $\bld{\gam^0}$}. Similar to the preceding 
two cases.

This completes the proof of Lemma \ref{lem:sixcon}.

\end{proof}
Since\, $\A(A^k_n)=\tfrac{\A(B^k_n)+\A(C^k_n)}2$,\, it follows that,
for both  $i=0, 1$, we have  $\frac{|\bet^i+\gam^i|}2=|\alp^i|=1$,
and hence  
\[
\bet^i=\gam^i=\alp^i\in S^{1}.
\]
On the other hand, the fact that  $F$ satisfies \eqref{eq:F} implies that
\[
\bet^1=\bet^0; \qquad  \gam^1=\theta\cdot\gam^0.
\]
We conclude that\, $\bet^1=\gam^1=\theta\cdot\gam^0=\theta\cdot\bet^0=\theta\cdot\bet^1$,
whence  $\theta=1$, in contradiction with the initial assumption that $\theta\neq1$.

The proof of Proposition \ref{prop:F} is complete.

\section{Two extensions of Theorem  \ref{thm:gwm}}\label{sec:2vers}
The following two theorems (Theorem \ref{thm:gwm2} and \ref{thm:gwm3})
extends Theorem \ref{thm:gwm} to arbitrary $f$-invariant ergodic 
measures $\mu$ (rather than  Lebesgue measure $\lam$). These results are of interest
in the case when  IETs $(X,F)$  are minimal but not uniquely ergodic.
(Such IETs exist, see \cite{Kea_nue}, \cite{KeNe}).
\begin{thm}\label{thm:gwm2}
Let  $(X,f)$  be a minimal $\mu$-ergodic IET, $X=[0,b)$,  where  $\mu$
is an \mbox{$f$-invariant} Borel probability measure. Then the set
\begin{equation}\label{eq:xwmmu}
X_{\wm}(\mu)=\{0<t<b\mid  f_{t}  \text{ is weakly mixing (relative $\mu$)}\,\}
\end{equation}
is a residual set of full $\mu$-measure: $\mu(X_{\wm}(\mu))=1$.
\end{thm}
\begin{proof}
Let $\bet\colon X\to X$ be an increasing 
continuous bijection taking measure $\mu$  to $\lam$. Then the composition  
$g=\bet^{\SSst-1}{\Sst\circ}\, f{\Sst\circ}\,\bet$\,  becomes a $\lam$-ergodic  IET
which topologically and combinatorially is $\bet$-isomoirphic
to $f$. (This is the essense of the normalization procedure discussed in 
\mbox{\cite[Section~1]{Ve_IET}}).

Theorem  \ref{thm:gwm} applies to $g$ to deduce Theorem \ref{thm:gwm2}.
\end{proof}

It is known that for any minimal  IET  $(X,f)$  the set\, $\MM_{\rm erg}(f)$ 
of ergodic $f$-invariant Borel probability measures on  $X$  is finite
(see \cite{Ka_fin} and \cite{Ve_IET}). 

\begin{thm}\label{thm:gwm3}
Let  $f\colon X\to X$  be a minimal IET, $X=[0,b)$. Then
\begin{align}\label{eq:xwmmu}
X_{\wm}(\MM_{\rm erg})=\{0<t<b\mid  f_{t}  \text{ is} & \text{ weakly mixing}\\ 
   &\text{ relative to every }\, \mu\in \MM_{\rm erg}(f)\,\}\notag
\end{align}
is a residual subset of $X$.
\end{thm}
\begin{proof}
Follows from Theorem \ref{thm:gwm2}  because 
\[
X_{\wm}(\MM_{\rm erg}(f))\,=
   \h{-3}{\bigcap_{\mu\in\MM_{\rm erg}(f)}\limits}\h{-2} X_{\wm}(\mu)
\]
is a finite intersection of residual sets.
\end{proof}

\section{Final Comments.}\label{sec:last}
The discussion in this section is conducted under the assumption that
$(X,f)$  is a fixed \lam-ergodic IET,  $X=[0,1]$,  so that both\, 
$X_{\nwm}$\, and\, $X_{\po}$ (the the exceptional and the 
critical sets, respectively) are defined
(by \eqref{eq:xnwm} and  \eqref{eq:crit}).

By Theorem \ref{thm:cwm},
every exceptional point is critical, i.e., the inclusion
\begin{equation}
X_{\nwm}(f)\!=\!X_{\nwm}\!\subset\!X_{\po}=\!X_{\po}(f)
\end{equation}
holds, and, by Theorem \ref{thm:mgg},  the critical set $X_{\po}$ is meager and has 
Lebesgue measure $0$. 

In this section we sketch some results on the size of the sets 
$X_{\nwm}$\, and\, $X_{\po}$, concerning their Hausdorff dimensions and 
cardinalities. The proofs and more detailed description of the results 
will appear elsewhere.

\subsection{Case: Irrational rotation} Let $a\in\R$ and set  $f(x)=R_a(x)=x+a \pmod1$.
(Such $f$ can be viewed as a $2$-IET). One verifies that in this case $f_{t}$
is a $3$-IET  provided that $t\neq D'=\{na\mid n\in\Z\}$ (see \eqref{eq:dp}).
Theorem \ref{thm:mgg} applies to deduce that $X_{\po}$ is meager and has 
Lebesgue measure $0$. 

It follows from \cite{BoNo} that 
\[
\big\{\Q+\Q\, a\big\}\cap (X\minus D')\cap X_{\po}=\emptyset,
\]
(i.e., no rational linear combination of $1$ and $a$  lying in  $X\minus D'$  
is exceptional). 

Both sets $X_{\nwm}(R_a)$\, and\, $X_{\po}(R_a)$  are countable if and only if $a$ is 
badly approximable (the sequences of partial quotients in its continued fraction 
expansion is bounded). This fact can be deduced from  \cite{Ve_69}.
In fact, if $a$ is badly approximable then the three sets,  
$X_{\nwm}$, $X_{\po}$ and $D'$,  coincide.

For Lebesgue almost all $a$, the Hausdorff dimensions of the sets  $X_{\nwm}(R_a)$\, 
and\, $X_{\po}(R_a)$  vanish. Nevertheless, for some irrational $a$, both Hausdorff 
dimensions can be 1. (This claim is based on an unpublished result by Yitwah Cheung). 

\subsection{Case: Linearly recurrent IETs} 
Every minimal $r$-IET  $f$  can be naturally coded by a minimal subshift $\Omega_{f}$ over the 
alphabet $\A_r=\{1,2,\ldots,r\}$ the following natural way described in \cite{Kea_IET}.
(Every point $x\in X$  corresponds to the infinite word $\big(W_x(k)\big)_{k\in\Z}\in (\A_r)^\Z$
determined by the rule:  $W_x(k)=s$ \  if \ $f^{k}(x)\in X_s=[d_{s-1}, d_s]$,
see \eqref{eq:disc}).

A minimal $r$-IET  $f$  is said to be linearly recurrent if the corresponding 
subshift $\Omega_{f}$ is linearly recurrent in the sense of \cite{Du} (see 
also \cite{DHS}).  Linearly recurrent IETs also appear in the literature as ``IETs of 
constant type'' (a characterization in terms of Rauzy-Veech induction).

Linearly recurrent IETs include IETs of periodic type (also known as 
pseudo-Anosov IETs) and, in particular, minimal IETs over quadratic number fields
(the IETs with the lengths of exchanged intervals lying in one and 
the same real quadratic number field). IETs over quadratic number fields
always reduce to Pseudo-Anosov IETs (see \cite{BoCa}).

One can prove that for linear recurrent IETs $(X,f)$ each of the sets $X_{\nwm}(f)$\, 
and\, $X_{\po}(f)$ is at most countable. 

We believe that for $r\geq4$ there are pseudo-Anosov $r$-IETs 
which are persitently weakly mixing (i.e., for which $X_{\nwm}(f)=\emptyset$).


\end{document}